\documentclass[a4paper,draft]{amsart}
\usepackage[a4paper,margin=25mm]{geometry}
\usepackage{amsmath}
\usepackage{amssymb,tikz,gauss}
\usepackage[colorlinks]{hyperref}

\parskip\smallskipamount

\def\<#1>{\langle#1\rangle}

\def\diag{\operatorname{diag}}
\def\oeis#1{\href{https://oeis.org/#1}{#1}}

\newtheorem{thm}{Theorem}
\newtheorem{lemma}{Lemma}
\newtheorem{prop}{Proposition}
\newtheorem{defn}{Definition}

\begin{document}

 \author[Robert Dougherty-Bliss]{Robert Dougherty-Bliss}
 \address{Robert Dougherty-Bliss, Rutgers University, New Brunswick NJ, USA}
 \email{robert.w.bliss@gmail.com}

 \author[Manuel Kauers]{Manuel Kauers\,$^\ast$}
 \address{Manuel Kauers, Institute for Algebra, J. Kepler University Linz, Austria}
 \email{manuel.kauers@jku.at}
 \thanks{$^\ast$ Supported by the Austrian FWF grants P31571-N32 and I6130-N}

 \title{Hardinian Arrays}

 \begin{abstract}
     In 2014, R.H.~Hardin contributed a family of sequences about king-moves on
     an arrays to the On-Line Encyclopedia of Integer Sequences (OEIS). The
     sequences were recently noticed in an automated search of the OEIS by
     Kauers and Koutschan, who conjectured a recurrence for one of them. We
     prove their conjecture as well as some older conjectures stated in the
     OEIS entries. We also have some new conjectures for the asymptotics of
     Hardin's sequences.
 \end{abstract}

 \maketitle

\section{Introduction}

The On-Line Encyclopedia of Integer sequences~\cite{sloane} contains over
350,000 sequences and perhaps tens of thousands of conjectures about them. Here
we resolve some of these conjectures related to a family of sequences due to
R.H.~Hardin.

For any positive integer $r$, let $H_r(n, k)$ be the number of $n \times k$
arrays which obey the following rules:
\begin{itemize}
\item The entry in position $(1, 1)$ is $0$, and the entry in position $(n,k)$
  is $\max(n, k) - r - 1$.
\item The entry in position $(i, j)$ must equal or be one more than each of
  the entries in positions $(i - 1, j)$, $(i, j - 1)$, and $(i - 1, j - 1)$.
\item The entry in position $(i, j)$ must be within $r$ of $\max(i, j) - 1$.
\end{itemize}
We call these arrangements of numbers Hardinian arrays. 
For $r = 1, 2, 3$, they are counted by the tables
\oeis{A253026}, \oeis{A253223}, and \oeis{A253004}, respectively.
Below is an example for $r=1$, $n=6$, and $k=5$.
\begin{equation*}
    \begin{bmatrix}
        0 & 1 & 2 & 2 & 3 \\
        1 & 1 & 2 & 2 & 3 \\
        2 & 2 & 2 & 3 & 3 \\
        3 & 3 & 3 & 3 & 4 \\
        4 & 4 & 4 & 4 & 4 \\
        4 & 4 & 4 & 4 & 4
    \end{bmatrix}
\end{equation*}
Hardin noticed several interesting patterns. For example, for every fixed $r$ and $k$,
the sequence $H_r(n,k)$ seems to be a polynomial in $n$ of degree $r$ for sufficiently
large~$n$. He also conjectured an evaluation of the diagonal for $r = 1$, namely
\[
    H_1(n, n) = \frac13(4^{n - 1} - 1).
\]
More recently, Kauers and Koutschan~\cite{kauers23d} performed an automated search for sequences
in the OEIS which satisfy linear recurrences with polynomial coefficients.
Hardin happened to submit the diagonal of $r = 2$ as its own sequence, which
led Kauers and Koutschan to conjecture a recurrence for $f(n) = H_2(n, n)$,
namely
\begin{alignat*}1
  &32 (n+1) (2 n+1)^2(1575 n^6+21285 n^5+117954 n^4+343020 n^3+551943 n^2+465785 n+161046) f(n)\\
  &-8 (121275 n^9+1933470 n^8+13267683 n^7+51280818 n^6+122556360 n^5+186866686 n^4\\
  &\qquad +180574335 n^3+105734340 n^2+33718283 n+4443102) f(n+1)\\
  &+2 (294525 n^9+4763070 n^8+33170868n^7+130145646 n^6+315713355 n^5+488415476 n^4\\
  &\qquad+478464380 n^3+283626704 n^2+91378536 n+12137328) f(n+2)\\
  &+(294525 n^9+4668570n^8+31877118 n^7+122735586 n^6+292620525 n^5+445804136 n^4\\
  &\qquad+431097970 n^3+252913504 n^2+80866406 n+10688508) f(n+3)\\
  &-(121275n^9+1961820 n^8+13655808 n^7+53503836 n^6+129484209 n^5+199650088 n^4\\
  &\qquad+194784258 n^3+114948300 n^2+36871922 n+4877748) f(n+4)\\
  &+2 (2 n+7) (1575 n^6+11835 n^5+35154 n^4+52554 n^3+41382 n^2+16118 n+2428) (n+3)^2 f(n+5)=0.
\end{alignat*}
Our main results are that many of these conjectures are correct. In
Section~\ref{sec:r1} we will prove Hardin's conjectured closed form for
$H_1(n,n)$ and extend this to a closed form for the rectangular case
$H_1(n,k)$. In Section~\ref{sec:rbig} we will prove that the conjectured
recurrence of Kauers and Koutschan for $H_2(n, n)$ is correct, and in fact that
\emph{every} $H_r(n, n)$ satisfies such a recurrence. We conjecture asymptotic
estimates for $H_r(n, n)$ for all $r \geq 2$.

\section{The case $r=1$}
\label{sec:r1}

This case can be settled by an elementary combinatorial argument. Let us first
consider the diagonal and confirm the closed form representation conjectured by
Hardin. In the following proof we index our arrays beginning from $0$ rather
than $1$.

\begin{thm}\label{thm:diag:r=1}
  $H_1(n,n)=\frac13(4^{n-1}-1)$ for all $n\geq1$.
\end{thm}

\begin{proof}
    Consider a valid $n \times n$ array. Above the upper diagonal, draw a
    dividing path between row entries which are equal to their king-distance
    and less than their king-distance. Draw the same path below the diagonal,
    but make it with respect to columns. See Figure~\ref{fig:double-path} for
    an example.

    By the monotonicty rule, the upper path can only move down and to the
    right. Further, if the first entry to its right in row $i$ is $(i, j)$,
    then the first entry to its right in row $i + 1$ is either $(i + 1, j)$ or
    $(i + 1, j + 1)$. Thus the upper-path essentially consists of two kinds of
    steps: down and right-down. The situation is mirrored in the lower path.

    If the upper path does not divide row $i$ just after the row's entry on the
    main diagonal, then the row is determined from the diagonal to the right
    endpoint. Entries between the diagonal and the path equal their
    king-distance, entries after the path equal one less than their
    king-distance, and the diagonal must equal $i$ as its king-distance is $i$
    and to its right is an $i + 1$. The analogous statement is true for the
    lower path with respect to columns. Thus every entry is determined except
    for when both paths divide the $i$th row and column just after the
    diagonal. In fact, the \emph{first} time this happens, the diagonal entry
    is still determined, as one of the entries above or to the left of the
    diagonal entry equals $i$.

    In summary, the only entries not determined by these paths are the diagonal
    entries which both paths are adjacent to, except the first one and last one
    (by rule). If one path first touches the diagonal at position $i$, and the
    other at position $j > i$, then there are $n - j - 2$ diagonal entries not
    determined. Of these entries, we may choose at most one to be the first
    less than its king-distance. After this choice all later entries must do
    the same. Thus each such pair of paths generates $n - j - 1$ valid arrays.

    If $C(k)$ is the number of paths which are first adjacent to the diagonal
    at position $k$, then
    \begin{equation*}
        H_1(n, n) = 2 \sum_{j = 0}^{n - 1} \sum_{i = 0}^{j - 1} C(i) C(j) (n - j - 1) + \sum_{j = 0}^{n - 1} C(j)^2 (n - j - 1).
    \end{equation*}
    Because each path essentially has two steps to choose from, both of them
    moving one step closer to their end, we have $C(k) = 2^{k - 1}$ if $k > 0$
    and $C(0) = 1$. Evaluating the above summations and simplifying produces
    $H_1(n, n) = (4^{n - 1} - 1) / 3$.
\end{proof}

The double-path idea used in the proof above extends to the case of
rectangular Hardinian arrays. The closed form expression for $H_1(n,k)$ shown
next confirms conjectures stated by Hardin for $H_1(n,1)$, $H_1(n,2)$, \dots, $H_1(n,7)$.

\begin{thm}
    \label{thm:rect:r=1}
  $H_1(n, k) = 4^{k - 1} (n - k) + \frac13(4^{k - 1} - 1)$ for all $n\geq k\geq1$.
\end{thm}

\begin{proof}
    Draw the same paths indicated in the proof of Theorem~\ref{thm:diag:r=1}.
    See Figure~\ref{fig:double-rect} for an example.

    A lower path now is either adjacent to the diagonal at some point or not.
    The number of valid arrays where the lower path is adjacent to the diagonal
    at some point is $H_1(k, k)$. All other pairs of paths contribute only one
    valid array. There are $n - k$ possible ending positions for a lower path
    which is never adjacent to the diagonal and $2^{k - 1}$ paths originating
    from each. Thus this case contributes $(n - k) 4^{k - 1}$ valid arrays.
    Together this yields $H_1(n, k) = 4^{k - 1} (n - k) + H_1(k, k)$.
\end{proof}

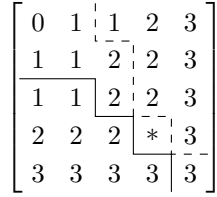
\begin{figure}
   \begin{tikzpicture}[scale=.5]
     \draw(-.1,-2.45)node {$\left[\rule{0pt}{1.4cm}\right.$};
     \draw(5.1,-2.45)node {$\left.\rule{0pt}{1.4cm}\right]$};
     \begin{scope}[yshift=-.5cm,xshift=.5cm]
       \draw(0,0) node {0} (1,0) node {1} (2,0) node {1} (3,0) node {2} (4,0) node {3}; 
       \draw[yshift=-1cm](0,0) node {1} (1,0) node {1} (2,0) node {2} (3,0) node {2} (4,0) node {3};
       \draw[yshift=-2cm](0,0) node {1} (1,0) node {1} (2,0) node {2} (3,0) node {2} (4,0) node {3};
       \draw[yshift=-3cm](0,0) node {2} (1,0) node {2} (2,0) node {2} (3,0) node {$\ast$} (4,0) node {3};
       \draw[yshift=-4cm](0,0) node {3} (1,0) node {3} (2,0) node {3} (3,0) node {3} (4,0) node {3};
     \end{scope}
     \draw[dashed] (2,0)--(2,-1)--(3,-1)--(3,-3)--(4,-3)--(4,-4)--(5,-4);
     \draw(0,-2)--(2,-2)--(2,-3)--(3,-3)--(3,-4)--(4,-4)--(4,-5);
   \end{tikzpicture}
   \caption{A generic $5\times 5$ matrix with two specific paths as constructed
       in the combinatorial proof of Theorem~\ref{thm:diag:r=1}. Every entry is
       determined by the paths except the one labeled $\ast$, which may be $3$
   or $2$.}
   \label{fig:double-path}
\end{figure}

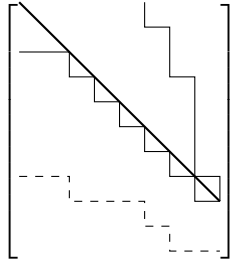
\begin{figure}
    \[
    \begin{bmatrix}
      \begin{tikzpicture}[scale=.33]
        \draw(5,8)--(5,7)--(6,7)--(6,5)--(7,5)--(7,1)--(8,1)--(8,0);
        \draw[thick] (0,8)--(8, 0);
        \draw (0,6)--(2,6)--(2,5)--(3,5)--(3,4)--(4,4)--(4,3)--(5,3)--(5,2)--(6,2)--(6,1)--(7,1)--(7,0)--(8,0);
        \begin{scope}[yshift=4cm]
          \draw [dashed] (0,-3)--(1,-3)--(2,-3)--(2,-4)--(4,-4)--(5,-4)--(5,-5)--(6,-5)--(6,-6)--(8,-6);
        \end{scope}
      \end{tikzpicture}
    \end{bmatrix}
    \]
    \caption{%
      The generic picture for paths in the proof of
      Theorem~\ref{thm:rect:r=1}. The lower two paths are examples of the two
  possible cases.}\label{fig:double-rect}
\end{figure}

As these combinatorial arguments do not seem to extend to $r>1$, we give
some alternative proofs of Theorem~\ref{thm:diag:r=1}. They all rely on the theorem
of Gessel and Viennot \cite[Theorem 10.13.1]{krattenthaler15}, which translates
the counting problem into a determinant evaluation problem. We will evaluate
the determinant in three different ways. The following notation will be used.

\begin{defn}
  \begin{enumerate}
  \item For each positive integer $n$, let $M(n)$ be the $n \times n$ matrix of
    binomial coefficients
    \[
        \left\{ \binom{u + v}{u} \right\}_{0 \leq u, v < n}.
    \]
    Observe that rows and columns are indexed starting from zero.     
  \item
    For any $n\times n$ matrix~$A$, any distinct row indices $i_1, i_2, \dots,i_r\in\{0,\dots,n-1\}$
    and distinct column indices $j_1, j_2, \dots, j_r\in\{0,\dots,n-1\}$, let
    $A_{i_1,i_2,\dots,i_r}^{j_1,j_2,\dots,j_r}$ be the $(n-r)\times(n-r)$ matrix
    obtained from $A$ by deleting rows $i_1,\dots,i_k$ and columns
    $j_1,\dots,j_k$.
  \item For every $n\geq1$, define
    \begin{align*}
        \Delta(n) &= \det M(n) \\
        \Delta(n)_{i_1,i_2,\dots,i_r}^{j_1,j_2,\dots,j_r} &=
            \det M(n)_{i_1,i_2,\dots,i_r}^{j_1,j_2,\dots,j_r}.
    \end{align*}
  \end{enumerate}
\end{defn}

 \begin{lemma} $\Delta(n)=1$ for all~$n$.
 \end{lemma}
 \begin{proof}
   Observe that $M(n)=AB$ where $A$ is the matrix whose entry at $(u,v)$ is $\binom uv$
   and $B$ is the matrix whose entry at $(u,v)$ is $\binom vu$.
   This follows from Vandermonde's identity $\binom{u+v}v=\sum_k\binom uk\binom vk$.
   As $A$ and $B$ are triangular matrices with $1$'s on the diagonal,
   the claim follows from $\Delta(n)=\det(M(n))=\det(A)\det(B)$.
 \end{proof}

 The key observation is that the valid $n \times n$ arrays can be partitioned
 into contiguous regions, as shown in Figure~\ref{fig:gesselviennot}.
 \begin{figure}
 \[
   \begin{bmatrix}
     \begin{tikzpicture}[scale=.5,yscale=-1]
     \begin{scope}[xshift=.5cm,yshift=.5cm]
       \draw(0,0)node{0}(1,0)node{1}(2,0)node{1}(3,0)node{2}(4,0)node{3}(5,0)node{4}(6,0)node{5};
       \draw[yshift=1cm](0,0)node{1}(1,0)node{1}(2,0)node{2}(3,0)node{2}(4,0)node{3}(5,0)node{4}(6,0)node{5};
       \draw[yshift=2cm](0,0)node{2}(1,0)node{2}(2,0)node{2}(3,0)node{2}(4,0)node{3}(5,0)node{4}(6,0)node{5};
       \draw[yshift=3cm](0,0)node{2}(1,0)node{2}(2,0)node{3}(3,0)node{3}(4,0)node{3}(5,0)node{4}(6,0)node{5};
       \draw[yshift=4cm](0,0)node{3}(1,0)node{3}(2,0)node{3}(3,0)node{3}(4,0)node{3}(5,0)node{4}(6,0)node{5};
       \draw[yshift=5cm](0,0)node{4}(1,0)node{4}(2,0)node{4}(3,0)node{4}(4,0)node{4}(5,0)node{4}(6,0)node{5};
       \draw[yshift=6cm](0,0)node{5}(1,0)node{5}(2,0)node{5}(3,0)node{5}(4,0)node{5}(5,0)node{5}(6,0)node{5};
     \end{scope}
     \draw(0,1)--(1,1)--(1,0);
     \draw(0,2)--(2,2)--(2,1)--(3,1)--(3,0);
     \draw(0,4)--(2,4)--(2,3)--(4,3)--(4,0);
     \draw(0,5)--(5,5)--(5,0);
     \draw(0,6)--(6,6)--(6,0);     
   \end{tikzpicture}
   \end{bmatrix}
   \]
   \caption{%
     The contiguous regions of a Hardinian array are separated by a tuple of nonintersecting lattice walks
     starting on the left and ending a the top.
   }\label{fig:gesselviennot}
  \end{figure}
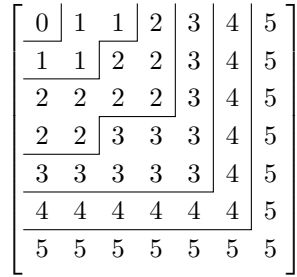%
There is a region for $0$, a region for $1$, a region for $2$, and so on. In
the $n \times n$ case, the region corresponding to $k$ is obtained by beginning
at the lowest occurrence of $k$ in the first column, moving as far right as
possible while only passing $k$'s, and moving up when stuck. For an $n \times
n$ Hardinian array this process always terminates in the first row.

\begin{prop}\label{prop:gesselviennot}
  $\displaystyle H_1(n, n) = \sum_{i = 0}^{n - 2} \sum_{j = 0}^{n - 2} \Delta(n-1)_i^j$ for all $n\geq1$.
\end{prop}

\begin{proof}
    The $n - 1$ contiguous regions in a Hardinian array of size $n \times n$
    are separated by $n - 2$ nonintersecting lattice paths. These paths begin
    on one of the $n - 1$ edges between entries in the first column and end on
    one of the $n - 1$ edges between entries in the first row, using only steps
    to the right ($\to$) and upwards ($\uparrow$). Each Hardinian array
    corresponds to exactly one such set of paths.

    In the other direction, each such set of paths corresponds to a Hardinian
    array. Given such a set, assign the induced regions the values $0$, $1$,
    \dots $n - 2$ in order from the top-left to the bottom-right. The top left
    will contain a $0$, the bottom right will contain an $n - 2$, and adjacent
    entries differ by no more than $1$. To see that the king-distance rule is
    not violated, note that it is not violated at the entries before the
    boundaries on the first column and first row---because at most one entry
    does not have a path just before it---and that these points have the
    largest king-distance of any entry reached using the available steps.

    It follows that the number of Hardinian arrays of size $n \times n$ equals
    the number of sets of nonintersecting lattice paths we have described. If
    we label the possible starting and ending positions $0, 1, \dots, n - 2$,
    then there are altogether $\binom{u+v}v$ paths from $u$ to~$v$, for any $u$
    and~$v$.

    Consider the set of paths where $i$ is the unique unchosen startpoint and
    $j$ the unique unchosen endpoint. In this case the $k$th path
    ($k=0,\dots,n-3$) starts at $k+[i\leq k]$ and ends at $k+[j\leq k]$. By the
    theorem of Gessel and Viennot, the number of such sets of paths is the
    determinant of the $(n-2)\times(n-2)$ matrix whose entry at position
    $(u,v)$ is $\binom{u+v+[i\leq u]+[j\leq v]}{v+[j\leq v]}$. This determinant
    equals $\Delta(n-1)_i^j$. It follows that $H_1(n, n)$ is the sum of
    $\Delta(n-1)_i^j$ over all possible rows~$i$ and columns~$j$.
\end{proof}

The proposition reduces the enumeration problem to the problem of evaluating
a sum of determinants. This can be done as follows. 

 \begin{proof}[Second proof of Theorem~\ref{thm:diag:r=1}]
   Let $\tilde M(n)$ be the $(n+1)\times(n+1)$ matrix obtained from $M(n)$ by first attaching an additional row $1,-1,1,-1,\dots$
   at the top and then an additional column $0,-1,1,-1,1,\dots$ at the left, e.g., 
   \[
   \tilde M(5) = \begin{vmatrix}
     0 & 1 & -1 & 1 & -1 & 1 \\
     -1 & 1 & 1 & 1 & 1 & 1 \\
     1 & 1 & 2 & 3 & 4 & 5 \\
     -1 & 1 & 3 & 6 & 10 & 15 \\ 
     1 & 1 & 4 & 10 & 20 & 35 \\
     -1 & 1 & 5 & 15 & 35 & 70
     \end{vmatrix}.
   \]
   By expanding along the first row and then along the first column, we have
   $\det\tilde M(n)=\sum_{i=0}^{n-1}\sum_{j=0}^{n-1}\Delta(n)_i^j$.

   It remains to determine the determinant of $\tilde M(n)$.

   Subtract the $(n-2)$nd row from the $(n-1)$st, then the $(n-3)$rd row from the $(n-2)$nd, and so on,
   and analogously for the columns, e.g.,
   \[
    \begin{gmatrix}[v]
     0 & 1 & -1 & 1 & -1 & 1 \\
     -1 & 1 & 1 & 1 & 1 & 1 \\
     1 & 1 & 2 & 3 & 4 & 5 \\
     -1 & 1 & 3 & 6 & 10 & 15 \\ 
     1 & 1 & 4 & 10 & 20 & 35 \\
     -1 & 1 & 5 & 15 & 35 & 70
     \rowops\add[-1]45\add[-1]34\add[-1]23\add[-1]12
     \colops\add[-1]45\add[-1]34\add[-1]23\add[-1]12
    \end{gmatrix}
    =\begin{vmatrix}
    0 & 1 & -2 & 2 & -2 & 2 \\
    -1 & 1 & 0 & 0 & 0 & 0 \\
    2 & 0 & 1 & 1 & 1 & 1 \\
    -2 & 0 & 1 & 2 & 3 & 4 \\
    2 & 0 & 1 & 3 & 6 & 10 \\
    -2 & 0 & 1 & 4 & 10 & 20     
    \end{vmatrix}
   \]
   In general, the proposed row and column operations replace the entry $\binom uv$ by
   \[
   \binom uv - \binom{u-1}v - (\binom u{v-1} - \binom{u-1}{v-1}) = \binom{u-1}{v-1}.
   \]
   Now expand along the second row (or column) to obtain
   \[
    \det\tilde M(n) = \Delta(n-1) + 4 \det\tilde M(n-1) = 4\det\tilde M(n-1) + 1
   \]
   for every~$n$. Together with the initial value $\det\tilde M(1)=1$, it follows
   by induction that $\det\tilde M(n)=\frac13(4^n-1)$.
   In view of Prop.~\ref{prop:gesselviennot}, Theorem~\ref{thm:diag:r=1} follows by replacing $n$ by $n-1$.
 \end{proof}

 \begin{proof}[Third proof of Theorem~\ref{thm:diag:r=1}]
   This proof uses computer algebra, in the spirit of an approach proposed by Zeilberger~\cite{zeilberger07}. 
   Because of $\Delta(n)=1$ and Cramer's rule, $(-1)^{i+j}\Delta(n)_i^j$ is the entry of $M(n)^{-1}$ at position
   $(i,j)$. For $n\geq1$ and $i,j=0,\dots,n-1$, define
   \[
   c(n,i,j)=(-1)^{i+j}\sum_{\ell=0}^{n-1}\binom i\ell\binom j\ell.
   \]
   Using symbolic summation algorithms (as implemented, e.g., in Koutschan's package~\cite{koutschan10c}),
   it can be easily shown that
   \[
     \sum_{k=0}^{n-1} \binom {i+k}k c(n,k,j) = \delta_{i,j}
   \]
   for all $n\geq1$ and all $i,j\geq0$. Therefore, $c(n,i,j)$ is the entry at $(i,j)$ of $M(n)^{-1}$,
   and thus equal to $(-1)^{i+j}\Delta(n)_i^j$.

   Applying summation algorithms once more, we can prove that the sum $s(n)=\sum_{i,j} (-1)^{i+j}c(n,i,j)$
   satisfies the recurrence 
   \[
   s(n+2)=5s(n+1)-4s(n)
   \]
   for all $n\geq1$. Together with the initial values $s(1)=1$ and $s(2)=5$, the claimed closed form expression
   now follows again by induction.
 \end{proof}

 While the sum $\Delta(n)_i^j=\sum_{\ell=0}^{n-1}\binom i\ell\binom j\ell$ 
 does not have a hypergeometric closed form,
 it does simplify in the special case $j=n-1$, where it turns out to be equal to $\binom{n-1}i$. 
 Taking the knowledge of this special case for granted, we can give a fourth proof of Theorem~\ref{thm:diag:r=1}.
 
 \begin{proof}[Fourth proof of Theorem~\ref{thm:diag:r=1}]
   Dodgson's identity (cf. Prop.~10 of Krattenthaler's tutorial on evaluating determinants~\cite{krattenthaler99}) says
   that 
   \[
     \det(A)\det(A_{i,n-1}^{j,n-1}) = \det(A_i^j)\det(A_{n-1}^{n-1}) - \det(A_i^{n-1})\det(A_{n-1}^j)
   \]
   for every $n\times n$ matrix~$A$. (Actually, Krattenthaler states the equation for $i=j=0$,
   but it is easily seen that it holds for arbitrary $i$ and $j$, because we can multiply $A$ with
   suitable permutation matrices from the left and the right in order to reduce to the case $i=j=0$.)

   Consider $A=M(n)$ and observe that $A_{n-1}^{n-1}=M(n-1)$. Then, because of $\Delta(n)=\Delta(n-1)=1$ it
   follows that
   \[
     \Delta(n-1)_i^j = \Delta(n)_i^j - \Delta(n)_i^{n-1}\Delta(n)_{n-1}^j.
   \]
   Using $\Delta(n)_i^{n-1}=\binom{n-1}i$ and $\Delta(n)_{n-1}^j=\binom{n-1}j$, it follows that
   \[
    \Delta(n)_i^j = \Delta(n-1)_i^j + \binom{n-1}i\binom{n-1}j.
   \]
   Summing over all $i$ and $j$ gives
   \[
     s(n) = s(n-1) + 4^{n-1},
   \]
   and with $s(1)=1$, the claim follows again by induction.
 \end{proof}
 
 \section{The case $r\geq2$}
 \label{sec:rbig}

 Via the theorem of Gessel and Viennot, we also have access to the sequences $H_r(n,n)$ for $r>1$.
 The argument is the same as for $r=1$, except that now a Hardinian array of size $n\times n$
 consists of $n-r$ contiguous regions, separated by $n-r-1$ nonintersecting lattice paths, whose
 start points and end points are taken from the set $\{0,\dots,n-2\}$. According to Gessel and
 Viennot, $\Delta(n-1)_{i_1,\dots,i_r}^{j_1,\dots,j_r}$ is the number of sets of $n-r-1$
 nonintersecting lattice walks whose start points are $\{0,\dots,n-2\}\setminus\{i_1,\dots,i_r\}$
 and whose end points are $\{0,\dots,n-2\}\setminus\{j_1,\dots,j_r\}$.

 In order to deal with these determinants, it helps to observe that Dodgson's identity quoted in
 the fourth proof of Theorem~\ref{thm:diag:r=1} is a special case of a more general identity due to
 Jacobi~\cite{jacobi33,rice07,abeles14}:
 For an $n\times n$ matrix $A$ and two choices $0\leq i_1<i_2<\dots<i_r<n$ and $0\leq j_1<j_2<\dots<j_r<n$
 of indices, form the $r\times r$ matrix $B$ whose entry at $(u,v)$ is defined as $\det(A_{i_u}^{j_v})$. 
 Then Jacobi's identity says
 \[
   \det(A)^{r-1}\det(A_{i_1,\dots,i_r}^{j_1,\dots,j_r}) = \det(B).
 \]
 For example, for $r=2$ we obtain
 \[
 \det(A)\det(A_{i_1,i_2}^{j_1,j_2}) =
 \begin{vmatrix}
   \det(A_{i_2}^{j_2}) & \det(A_{i_2}^{j_1}) \\
   \det(A_{i_1}^{j_2}) & \det(A_{i_1}^{j_1}) 
 \end{vmatrix}
 = \det(A_{i_1}^{j_1})\det(A_{i_2}^{j_2}) - \det(A_{i_1}^{j_2})\det(A_{i_2}^{j_1}),
 \]
 and setting $i_2=j_2=n-1$ gives Dodgson's version. 

 \begin{thm}\label{thm:diag:r=2}
   For every $r\geq2$, the sequence $H_r(n,n)$ is D-finite.
   In particular, the sequences \oeis{A253217} ($r=2$) and \oeis{A252998} ($r=3$) are D-finite.
 \end{thm}
 \begin{proof}
   For $A=M(n)$, Jacobi's identity implies
   \[
   \Delta(n)_{i_1,\dots,i_r}^{j_1,\dots,j_r} =
   \begin{vmatrix}
     \Delta(n)_{i_1}^{j_1} & \cdots & \Delta(n)_{i_1}^{j_r} \\
     \vdots & \ddots & \vdots \\
     \Delta(n)_{i_r}^{j_1} & \cdots & \Delta(n)_{i_r}^{j_r}
   \end{vmatrix}
   \]
   For every fixed~$r$, the determinant on the right is D-finite because it depends
   polynomially on quantities which we have recognized in the previous section as
   being D-finite. It follows that the left hand side is D-finite, and consequently, 
   \[
   H_r(n,n) = \sum_{0\leq i_1<\dots<i_r\leq n-2}
              \sum_{0\leq j_1<\dots<j_r\leq n-2} \Delta(n-1)_{i_1,\dots,i_r}^{j_1,\dots,j_r}
   \]
   is D-finite, too.
 \end{proof}

 Theorem~\ref{thm:diag:r=2} is not quite enough to confirm the correctness of the recurrence equation
 Kauers and Koutschan obtained for $H_2(n,n)$ via guessing~\cite{kauers23d}.
 The theorem only implies that the sequence satisfies \emph{some} recurrence.
 In order to explicitly construct a recurrence, we have to evaluate the two 6-fold sums
 \begin{alignat*}1
   S_1(n)
   &=\sum_{i_1\geq0}\sum_{i_2>i_1}\sum_{j_1\geq0}\sum_{j_2>j_1}\sum_{u=0}^n\sum_{v=0}^n\binom u{i_1}\binom u{j_1}\binom v{i_2}\binom v{j_2}\\
   &=\sum_{u=0}^n\sum_{v=0}^n
   \biggl(\underbrace{\sum_{i_1\geq0}\sum_{i_2>i_1}\binom u{i_1}\binom v{i_2}}_{=:s(u,v)}\biggr)
   \biggl(\underbrace{\sum_{j_1\geq0}\sum_{j_2>j_1}\binom u{j_1}\binom v{j_2}}_{=s(u,v)}\biggr)\text{ and}\\
   S_2(n)&=
   \sum_{i_1\geq0}\sum_{i_2>i_1}\sum_{j_1\geq0}\sum_{j_2>j_1}\sum_{u=0}^n\sum_{v=0}^n\binom u{i_1}\binom u{j_2}\binom v{i_2}\binom v{j_1}\\
   &=\sum_{u=0}^n\sum_{v=0}^n
   \biggl(\underbrace{\sum_{i_1\geq0}\sum_{i_2>i_1}\binom u{i_1}\binom v{i_2}}_{=s(u,v)}\biggr)
   \biggl(\underbrace{\sum_{j_1\geq0}\sum_{j_2>j_1}\binom v{j_1}\binom u{j_2}}_{=s(v,u)}\biggr).
 \end{alignat*}
 It seems best to do this using generating functions. We have
 \[
  \sum_{u=0}^\infty\sum_{v=0}^\infty s_1(u,v) x^uy^v=\frac{y}{(1-x-y)(1-2y)}.
 \]
 The generating functions of $s(u,v)^2$ and $s(u,v)s(v,u)$ can be expressed as Hadamard products.
 As explained in~\cite{bostan16b}, Hadamard products can be rephrased as residues, and residues can be
 computed via creative telescoping~\cite{zeilberger90}.
 Using Koutschan's implementation~\cite{koutschan10c}, it is easy to prove
 \begin{alignat*}1
   \frac{y}{(1-x-y)(1-2y)} \odot_{x,y} \frac{y}{(1-x-y)(1-2y)} &= \frac{y}{2 x+2 y-1} \left(\frac{1}{\sqrt{x^2-2 x (y+1)+(y-1)^2}}+\frac{2}{4 y-1}\right)\\
   \frac{y}{(1-x-y)(1-2y)} \odot_{x,y} \frac{x}{(1-x-y)(1-2x)} &= \frac1{2 (2 x+2 y-1)} \left(\frac{x+y-1}{\sqrt{x^2-2 x (y+1)+(y-1)^2}}+1\right),
 \end{alignat*}
 respectively.
 Summing $u$ from $0$ to $n$ and $v$ from $0$ to $m$ amounts to multiplying these series by $\frac1{(1-x)(1-y)}$,
 and setting $m$ to $n$ amounts to taking the diagonals of the resulting bivariate series:
 \begin{alignat*}1
   &\diag \frac1{(1-x)(1-y)} \frac{y}{2 x+2 y-1} \left(\frac{1}{\sqrt{x^2-2 x (y+1)+(y-1)^2}}+\frac{2}{4 y-1}\right), \\
   &\diag \frac1{(1-x)(1-y)} \frac1{2 (2 x+2 y-1)} \left(\frac{x+y-1}{\sqrt{x^2-2 x (y+1)+(y-1)^2}}+1\right),
 \end{alignat*}
 respectively.
 As diagonals can also be rephrased as residues (cf. again~\cite{bostan16b} for a detailed discussion),
 we can apply creative telescoping to obtain linear differential operators annihilating these series.
 Their least common left multiple is an annihilator of the generating function of~$H_2(n,n)$.

 In the end, we obtained a linear differential operator of order~10 with polynomial coefficients of degree~43.
 With this certified operator at hand, we can prove that the guessed recurrence of Kauers and Koutschan is
 correct.

 In principle, we could derive a recurrence for $H_r(n,n)$ for any $r\geq2$ in the same way, but already
 for $r=3$ the computations become too costly.
 We can however use the formula
 \[
   H_r(n,n) = \sum_{0\leq i_1<\dots<i_r\leq n-2} \,
              \sum_{0\leq j_1<\dots<j_r\leq n-2} \Delta(n-1)_{i_1,\dots,i_r}^{j_1,\dots,j_r}
 \]
 to compute some more terms of the sequences. In order to do this efficiently, we can recycle the idea
 of the second proof of Theorem~\ref{thm:diag:r=1} and translate some of the summation signs into additional
 rows and columns of the determinant. For example, for $r=3$ we have
 \[
   H_r(n,n) = \sum_{i=0}^{n-2}\sum_{j=0}^{n-2} |\det(A_{i,j})|
 \]
 where $A_{i,j}$ is the matrix obtained from $M(n-1)$ by removing the $i$th row and the $j$th column and
 adding a row with alternating signs in the column range $0\dots j-1$ followed by zeros and an additional
 row with zeros in the column range $0\dots j-1$ followed by alternating signs; and similarly two additional
 columns. For example, for $n=8, i=4, j=5$ we have
 \begin{center}
   \begin{tikzpicture}[xscale=.65,yscale=.45]
     \draw(-1.3,-3.5)node{$A_{i,j}=\left(\rule{0pt}{1.9cm}\right.$}
          (7.6,-3.5)node{$\left.\rule{0pt}{1.9cm}\right)$};
                  \draw(0,0)node{$0$}(1,0)node{$0$}(2,0)node{$0$}(3,0)node{$0$}(4,0)node{$0$}(5,0)node{$0$}(6,0)node{$1$}(7,0)node{$-1$};
     \draw[yshift=-1cm](0,0)node{$0$}(1,0)node{$0$}(2,0)node{$-1$}(3,0)node{$1$}(4,0)node{$-1$}(5,0)node{$1$}(6,0)node{$0$}(7,0)node{$0$};
     \draw[yshift=-2cm](0,0)node{$0$}(1,0)node{$-1$}(2,0)node{$1$}(3,0)node{$1$}(4,0)node{$1$}(5,0)node{$1$}(6,0)node{$1$}(7,0)node{$1$};
     \draw[yshift=-3cm](0,0)node{$0$}(1,0)node{$1$}(2,0)node{$1$}(3,0)node{$2$}(4,0)node{$3$}(5,0)node{$4$}(6,0)node{$6$}(7,0)node{$7$};
     \draw[yshift=-4cm](0,0)node{$0$}(1,0)node{$-1$}(2,0)node{$1$}(3,0)node{$3$}(4,0)node{$6$}(5,0)node{$10$}(6,0)node{$21$}(7,0)node{$28$};
     \draw[yshift=-5cm](0,0)node{$-1$}(1,0)node{$0$}(2,0)node{$1$}(3,0)node{$5$}(4,0)node{$15$}(5,0)node{$35$}(6,0)node{$126$}(7,0)node{$210$};
     \draw[yshift=-6cm](0,0)node{$1$}(1,0)node{$0$}(2,0)node{$1$}(3,0)node{$6$}(4,0)node{$21$}(5,0)node{$56$}(6,0)node{$252$}(7,0)node{$462$};
     \draw[yshift=-7cm](0,0)node{$-1$}(1,0)node{$0$}(2,0)node{$1$}(3,0)node{$7$}(4,0)node{$28$}(5,0)node{$84$}(6,0)node{$462$}(7,0)node{$924$};
     \draw(1.5,.5)--+(0,-8) (-.5,-1.5)--+(8,0)
     (.5,-7.5)node[below]{\vbox{\clap{\scriptsize extra\mathstrut}\kern-5pt\clap{\scriptsize columns\mathstrut}}}
     (7.7,-.5)node[right]{\vbox{\rlap{\scriptsize extra\mathstrut}\kern-5pt\rlap{\scriptsize rows\mathstrut}}};
     \draw[dotted]
     (-.5,-4.5)--+(8,0) (7.7,-4.5)node[right]{\vbox{\rlap{\scriptsize $i$th row\mathstrut}\kern-5pt\rlap{\scriptsize deleted\mathstrut}}}
     (5.5,.5)--+(0,-8) (5.5,-7.5)node[below]{\vbox{\clap{\scriptsize $j$th column\mathstrut}\kern-5pt\clap{\scriptsize deleted\mathstrut}}};
   \end{tikzpicture}
 \end{center}
 With this optimiziation, it is not difficult to compute the first 100 terms, and using these, the technique
 of~\cite{kauers22b} is able to guess a convincing recurrence equation of order~9 and degree~36. It is not
 reproduced here. 

 For $r=4$, we explicitly delete two rows and columns and add two rows and columns with alternating signs, as shown
 in Figure~\ref{fig:r=4,r=5} on the left.
 \begin{figure}
 \begin{center}
   \begin{tikzpicture}[scale=.5]
     \draw(0,0) rectangle (6,-6);
     \draw(1,0)--(1,-6) (0,-1)--(6,-1) (.5,-6)node[below]{\scriptsize extra\mathstrut} (6,-.5)node[right]{\scriptsize extra\mathstrut};
     \draw[dotted](2.5,0)--(2.5,-6) node[below]{\scriptsize$j_1\mathstrut$} (4,0)--(4,-6) node[below]{\scriptsize$j_2\mathstrut$}
     (0,-2.5)--(6,-2.5)node[right]{\scriptsize$i_1\mathstrut$} (0,-5)--(6,-5)node[right]{\scriptsize$i_2\mathstrut$};
     \draw[very thick, dotted](1,-.75)--(2.5,-.75) (2.5,-.25)--(4,-.25) (.25,-2.5)--(.25,-5) (.75,-1)--(.75,-2.5);
   \end{tikzpicture}\hfil
   \begin{tikzpicture}[scale=.5]
     \draw(-.5,.5) rectangle (6,-6);
     \draw(1,.5)--(1,-6) (-.5,-1)--(6,-1) (.33,-6)node[below]{\scriptsize extra\mathstrut} (6,-.33)node[right]{\scriptsize extra\mathstrut};
     \draw[dotted](2.5,-.5)--(2.5,-6) node[below]{\scriptsize$j_1\mathstrut$} (4,-.5)--(4,-6) node[below]{\scriptsize$j_2\mathstrut$}
     (-.5,-2.5)--(6,-2.5)node[right]{\scriptsize$i_1\mathstrut$} (-.5,-5)--(6,-5)node[right]{\scriptsize$i_2\mathstrut$};
     \draw[very thick, dotted](1,-.75)--(2.5,-.75) (2.5,-.25)--(4,-.25) (4,.25)--(6,.25)
     (.25,-2.5)--(.25,-5) (.75,-1)--(.75,-2.5) (-.25,-5)--(-.25,-6);
   \end{tikzpicture}
 \end{center}
 \caption{%
   Left: the computation of $\sum_{i_1<i_2<i_3<i_4}\sum_{j_1<j_2<j_3<j_4}\Delta(n-1)_{i_1,i_2,i_3,i_4}^{j_1,j_2,j_3,j_4}$ is equivalent
   to the computation of the sum over $i_1,i_2$ and $j_1,j_2$ of the determinants constructed as shown in the figure.
   Light dots indicated omitted rows and columns; strong dots indicate regions filled with alternating signs.
   \hfill\break
   Right: the computation of $\sum_{i_1<i_2<i_3<i_4<i_5}\sum_{j_1<j_2<j_3<j_4<j_5}\Delta(n-1)_{i_1,i_2,i_3,i_4,i_5}^{j_1,j_2,j_3,j_4,j_5}$ is equivalent
   to the computation of the sum over $i_1,i_2$ and $j_1,j_2$ of the determinants constructed as shown in the figure.
 }\label{fig:r=4,r=5}
 \end{figure}
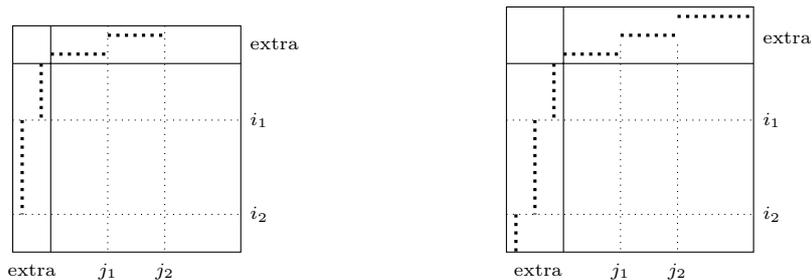%
 This allows us to reduce the original 8-fold sum to a 4-fold sum.
 A 4-fold sum is also sufficient for $r=5$, where we can even eliminate six summations by adding extra rows and columns,
 as shown in Figure~\ref{fig:r=4,r=5} on the right. 
 By computing the sums over all these determinants, we were able to determine the first $\approx65$ terms of the
 sequences $H_4(n,n)$ and $H_5(n,n)$. Unfortunately, these terms were not sufficient to find a recurrence by guessing.

 However, the terms are enough to obtain convincing conjectured expressions for their asymptotics.
 We obtained the following conjectures:
 \begin{center}
   \begin{tabular}{c|c|l}
     $r$ & asymptotics & remark \\\hline
     0 & 1 & trivial \\
     1 & $\frac1{2^23}\,{4^n}$ & by Theorem~\ref{thm:diag:r=1} \\
     2 & $\frac{1}{2^23^4\pi}\,16^n\,n^{-1}$ & from the proven recurrence \\
     3 & $\frac{1}{2^23^9\pi}\,64^n\,n^{-3}$ & from the guessed recurrence \\
     4 & $\frac{2^2}{3^{16}\pi^2}\,256^n\,n^{-6}$ & from the first 70 terms \\
     5 & $\frac{2^4}{3^{23}\pi^2}\,1024^n\,n^{-10}$ & from the first 70 terms
   \end{tabular}
 \end{center}
 Altogether, it seems that for every $r\geq0$, we have
 \[
 H_r(n,n)\sim c\,2^{2rn} n^{-\binom r2}\qquad(n\to\infty)
 \]
 for some constant $c$ that can be expressed as a power product of $2$, $3$, and~$\pi$.

 At least for specific values of~$r$, it might be possible to prove these conjectured asymptotic formulas
 using the powerful techniques of analytic combinatorics in several variables~\cite{pemantle13,melczer21}.
 However, in order to invoke these techniques, we would need to know more about the bivariate
 sequences~$H_r(n,k)$. Unfortunately, while we found an explicit expression for $H_1(n,k)$, we were not
 able to show that $H_r(n,k)$ is D-finite as a bivariate sequence in $n$ and $k$ for any $r\geq2$, although
 we suspect it to be.  
 
 \bibliographystyle{plain}
 \bibliography{all}
 
\end{document}